\documentclass[12pt]{amsart}
\usepackage[utf8]{inputenc}
\usepackage[T1]{fontenc}
\usepackage{amsthm,amsmath,amssymb}
\usepackage{hyperref}
\usepackage{cleveref}
 \newtheorem{Thm}{Theorem}[section]
 \newtheorem{Lem}[Thm]{Lemma}
 \newtheorem{Prop}[Thm]{Proposition}
 \newtheorem{Cor}[Thm]{Corollary}
 
\theoremstyle{remark}

\newtheorem{Qu}[Thm]{Question}
\theoremstyle{definition}
 
\numberwithin{equation}{section}

\newcommand\inv{^{-1}}
\def\HM#1.#2.#3.#4.{{^{#1}_{#3}\mathcal M^{#2}_{#4}}}

\newcommand\Vect{\operatorname{Vect}}

\newcommand\Mod{\operatorname{Mod}}

\newcommand\Q{\mathbb Q}
\newcommand\CC{\mathbb C}
\newcommand\ZZ{\mathbb Z}

\newcommand\C{\mathcal C}
\newcommand\CTR{\mathcal Z}

\newcommand\semidir\rtimes

\newcommand\Inf{\operatorname{Inf}}
\newcommand\Gal{\operatorname{Gal}}
\newcommand\ab{{\operatorname{ab}}}
\begin{document}
\title[Equivalent modular data]{Modular categories are not determined by their modular data}
\author{Michaël Mignard}
\email{michael.mignard@u-bourgogne.fr}
\author{Peter Schauenburg}
\email{peter.schauenburg@u-bourgogne.fr}
\address{Institut de Math{\'e}matiques de Bourgogne, UMR 5584 CNRS
\\
Universit{\'e} Bourgogne Franche-Comté\\
F-21000 Dijon\\France}


\begin{abstract}
Arbitrarily many pairwise inequivalent modular categories can share the same modular data. We exhibit a family of examples that are module categories over twisted Drinfeld doubles of finite groups, and thus in particular integral modular categories.
\end{abstract}
\maketitle

\section{Introduction}
\label{sec:introduction}

A modular category is a braided spherical fusion category such that the square matrix $S$ whose coefficients are the traces of the square of the braiding on pairs of simple objects is invertible. This S-matrix, together with the T-matrix, which is diagonal and contains the values of a kink on the simple objects, defines a projective representation of the modular group on the complexified Grothendieck ring of the category. The two matrices together are referred to as the modular data of the modular category. 

The modular data constitute a very important invariant of modular categories, both powerful for the structure theory and classification of such categories, and central to their role in mathematical physics and low-dimensional topology. For the theory of modular categories, and their background in mathematical physics, we refer generally to \cite{BakKir:LTCMF}. 

It is a natural question whether the modular data constitute a complete invariant of modular tensor categories, or at least modular tensor categories in some interesting class. The authors of \cite{MR3486174} expressed their belief that this is true for integral modular categories; it is the present authors' impression, however, that the research community in general, including perhaps the authors of \cite{MR3486174}, did not have great faith in that conjecture. Still, it seems that to this date no explicit example was known of two modular categories that share the same modular data without being equivalent as modular categories. In recent computer-based work, the authors have provided evidence rather in favor of the conjecture: The twisted Drinfeld doubles of finite groups up to order $31$ which have the same modular data are equivalent; altogether, there are $1126$ inequivalent nontrivial modular categories that arise in this fashion. In fact they are already distinguished by the coefficients of the T-matrices taken together with the values of the higher Frobenius-Schur indicators for simple objects; this is a potentially weaker invariant in general.

The present paper now exhibits a family of examples showing in fact that arbitrarily many modular categories can share the same modular data without being equivalent as braided tensor categories. Moreover, our examples are found among the particularly accessible class of module categories over twisted Drinfeld doubles of finite groups. In our examples the categories in each family either have identical modular data, or they are distinguished by their T-matrices.

\section{Preliminaries}
\label{sec:prelims}

Throughout the paper we consider fusion categories in the sense of \cite{MR2183279} defined over the field of complex numbers. A general reference for the theory of fusion categories is \cite{MR3242743}. We denote $\Vect_G^\omega$ the category of finite dimensional $G$-graded vector spaces with associativity constraint given by a normalized three-cocycle $\omega\colon G^3\to \C^\times$. We note that $G$ is determined by the fusion category $\Vect_G^\omega$ as its Grothendieck ring, and that $\Vect_G^\omega$ and $\Vect_G^{\nu}$ are equivalent as fusion categories if and only if there is an automorphism $f$ of $G$ such that $\nu$ and $f^*\omega$ are cohomologous.

A \emph{group-theoretical fusion category} is a fusion category which is categorically Morita equivalent to a pointed fusion category, see \cite{MR1976459}, and \cite{MR3077244} for an overview, or \cite[9.7]{MR3242743}. Any group-theoretical fusion category categorically Morita equivalent to $\Vect_G^\omega$ can be described in terms of a couple $(H,\mu)$ where $H$ is a subgroup of $G$ and $\mu\colon H\times H\to \CC^\times$ is a two-cochain  with $d\mu=\omega|_{H\times H\times H}$. The last condition ensures that the twisted group algebra $\CC_\mu[H]$, which has $H$ as a basis and multiplication $x\cdot y=\mu(x,y)xy$ for $x,y\in H$, is an associative algebra in the fusion category $\Vect_G^\omega$. The associated group-theoretical fusion category $\C(G,H,\omega,\mu)$ is the category of endofunctors of the category of $\CC_\mu[H]$-modules in $\Vect_G^\omega$. A slightly more elementary (but equivalent by Watts' theorem) description is that $\C(G,H,\omega,\mu)$ is the category of bimodules in $\Vect^\omega_G$ over the algebra $\CC_\mu[H]$, with the tensor product of bimodules taken over $\CC_\mu[H]$.

We note several elementary types of equivalences between group-theoretical fusion categories: If $f\colon G'\to G$ is an isomorphism of groups, then $\C(G,H,\omega,\mu)$ and $\C(G',f\inv(H),f^*\omega,f^*\mu)$ are equivalent. If $\theta\colon G^2\to\C^\times$ is a cochain, then $\C(G,H,\omega,\mu)$ and $\C(G,H,\omega d\theta,\mu\theta)$ are equivalent. In particular, $\C(G,H,\omega,\mu)$ is equivalent to $\C(G,H,\omega',1)$ where $\omega'=\omega d\tilde\theta\inv$ with any extension $\tilde\theta$ of $\theta$ to $G\times G$. If $\lambda\colon H\to\C^\times$ is a cochain, then $\C(G,H,\omega,\mu)$ is equivalent to $\C(G,H,\omega,\mu d\lambda)$.

 Two fusion categories are categorically Morita equivalent if and only if their Drinfeld centers are equivalent braided monoidal categories \cite{MR2183279}. The Drinfeld center of $\Vect_G^\omega$ is equivalent to the module category of the twisted Drinfeld double $D_\omega(G)$ introduced in \cite{MR1128130}; the modular data of these modular categories is the focus of \cite{MR1770077}.
\section{Pointed fusion categories of rank $pq$}
\label{sec:pointed}

Let $p,q$ be odd primes such that $p|q-1$. It is well-known that there exist exactly two non-isomorphic groups of order $pq$, namely the cyclic group and a nonabelian group, denoted $G$ in the sequel, which is a semidirect product $\ZZ_q\semidir \ZZ_p=\ZZ_q\semidir_n \ZZ_p$ in which the generator of $\ZZ_p$ acts on $\ZZ_q$ as multiplication by an element $n$ of multiplicative order $p$ in $\ZZ_q$. Let $\kappa\colon\ZZ_p^3\to\CC^\times$ denote a representative of a generator of the cohomology group $H^3(\ZZ_p,\CC^\times)\cong \ZZ_p$. We can assume that $\kappa$ takes values in $p$-th roots of unity.

We note for later use that the automorphism of $\ZZ_p$ given by multiplication with $m\in\ZZ$ (prime to $p$) induces the action on $H^3(\ZZ_p,\CC^\times)\cong\ZZ_p$ given by multiplication with $m^2$. This can be deduced by a direct calculation with the well-known periodic $\ZZ_p$-resolution of $\ZZ$ \cite[IV.7]{MR1344215}. A more conceptual proof uses the structure of the integral cohomology ring of $\ZZ_p$ and the fact that the group has periodic cohomology, cf.~\cite{MR0077480,MR672956}. In fact $H^3(\ZZ_p,\CC^\times)\cong H^4(\ZZ_p,\ZZ)$, and an isomorphism between $H^k(\ZZ_p,\ZZ)$ and $H^{k+2}(\ZZ_p,\ZZ)$ is given by multiplication with a generator $\eta$ of $H^2(\ZZ_p,\ZZ)\cong \ZZ_p$ in the cohomology ring. Thus $H^4(\ZZ_p,\ZZ)\cong\ZZ_p$ is generated by $\eta^2$; the action of $m$ on $H^2(\ZZ_p,\ZZ)$ is given by the action on $\ZZ_p$, and thus the action on $H^4(\ZZ_p,\ZZ)$ is given by $m^2$. 

Now we can define $p$ pointed fusion categories $\C_0,\dots\C_{p-1}$ by $\C_u=\Vect_G^{\omega^u}$ where $\omega=\Inf_{\ZZ_p}^G\kappa$ is the inflation of $\kappa$ to $G$.
\begin{Lem}\label{Lem:pStueck}
  The categories $\C_0,\dots,\C_{p-1}$ are pairwise not categorically Morita equivalent. Equivalently, the Drinfeld centers $\CTR(\C_0),\dots,\CTR(\C_{p-1})$ are pairwise inequivalent braided monoidal categories.
\end{Lem}
\begin{proof}
  Any automorphism of $G$ fixes the unique subgroup of order $q$ and induces the identity on the quotient $\ZZ_p$. Thus, the $p$ pointed categories are pairwise inequivalent fusion categories. The only fusion categories categorically Morita equivalent to $\C_u$ are the group-theoretical fusion categories $\C(G,H,\widetilde\omega,\mu)$ where $\widetilde\omega$ is a three-cocycle on $G$ mapped to $\omega^u$ by an automorphism of $G$, $H$ is a subgroup of $G$ and $\mu$ is a two-cochain on $H$ whose coboundary is the restriction of $\widetilde\omega$. For $\C(G,H,\widetilde\omega,\mu)$ to be pointed, it is necessary (see e.g.~\cite{MR2362670}) that $H$ be abelian and normal in $G$. Now the only three-cocycle mapped to $\omega^u$ by an automorphism of $G$ is $\omega^u$ itself. The only nontrivial abelian normal subgroup is $H=\ZZ_q$, and the restriction of $\omega^u$ to $H$ is trivial, so $\mu$ has to be a two-cocycle. Since the second cohomology group of $H$ is trivial, we can assume that $\mu$ is trivial outright. In this case the category $\C(G,H,\omega^u,1)$ is explicitly known \cite[Thm.2.1]{MR2333187} in the form $\Vect_{G'}^\alpha$ and can be described as follows: $G'$ is the semidirect product of $\ZZ_p$ with the dual group of $\ZZ_q$ under the dualized action, and $\alpha=\Inf_{\ZZ_p}^{G'}\kappa^u$. If we identify $\ZZ_q$ with its dual, this means that $G'=\ZZ_q\semidir_{n'}\ZZ_p$, with the generator of $\ZZ_p$ acting as multiplication with the multiplicative inverse $n'=n^{p-1}$ of $n$ modulo $q$. This in turn means that $G\cong G'$, with the isomorphism inducing the inversion map on the quotient $\ZZ_p$. But the inversion map, or multiplication by $-1$, acts as the identity on $H^3(\ZZ_p,\CC^\times)$, so the isomorphism pulls $\Inf_{\ZZ_p}^{G'}\kappa^u$ back to $\omega^u$, and we conclude that $\C(G,H,\omega^u,1)$ is equivalent to $\C_u$ as a fusion category.
\end{proof}

For the primary purpose of the present paper having found the $p$ distinct categories above is sufficient --- we will see in the next section that they afford only three distinct modular data. For completeness we will continue and classify all pointed fusion categories of dimension $pq$ up to categorical Morita equivalence:

\begin{Lem}
  Let $p<q$ be primes. Then two pointed fusion categories of rank $pq$ are categorically Morita equivalent if and only if they are equivalent monoidal categories.
  \begin{enumerate}
  \item If $p$ is odd and $p\mid q-1$, then there are exactly $p+9$ equivalence classes of pointed fusion categories of dimension $pq$, namely the $p$ categories defined above, and nine categories of the form $\Vect_{\ZZ_p\times\ZZ_q}^\alpha$.
  \item If $p$ is odd and $p\nmid q-1$, there are exactly nine equivalence classes of pointed fusion categories of dimension $pq$.
  \item If $p=2$, then there are exactly $12$ equivalence classes of pointed fusion categories of dimension $pq$.
  \end{enumerate}
\end{Lem}
\begin{proof}
  We have $H^3(\ZZ_p,\CC^\times)\cong\ZZ_p$, and the automorphism of $\ZZ_p$ given by multiplication with $x\in\ZZ_p^\times$ acts on the cohomology group as multiplication with $x^2$; thus there are three orbits, namely zero, the quadratic residues, and the quadratic nonresidues (except when $p=2$, where there are two orbits). The same holds for $q$, and thus $H^3(\ZZ_p\times\ZZ_q,\CC^\times)\cong \ZZ_p\times\ZZ_q$ has $3\cdot 3=9$ orbits under the action of the automorphism group of $\ZZ_p\times\ZZ_q$ if $p>2$, and $6$ if $p=2$. It is easy to see that the nine, resp. six pointed fusion categories thus obtained are categorically Morita equivalent only to themselves.
 
 Now let $p|q-1$. The cohomology groups of the nonabelian group $G$ of order $pq$ (are surely well-known and) can be readily computed using \cite[Cor.10.8]{MR1344215}: For the unique subgroup $H$ of order $q$ there is a short exact sequence
  \begin{equation*}
    0\rightarrow H^3(\ZZ_p,\CC^\times)\to H^3(G,\CC^\times)\to H^3(\ZZ_q,\CC^\times)^{\ZZ_p}\to 0.
  \end{equation*}
If the generator of $\ZZ_p$ acts on $\ZZ_q$ as multiplication by $x$, it acts on $H^3(\ZZ_q,\CC^\times)$ as multiplication by $x^2$. Thus, if $p>2$,  the action is nontrivial,  the invariants are trivial, and every cohomology class is inflated from $\ZZ_p$, giving the $p$ categories we have already studied. 

If $p=2$ (so $G$ is a dihedral group), then the action is trivial, thus $H^3(G,\CC^\times)\cong \ZZ_2\times\ZZ_q$. Any automorphism of $\ZZ_q$ extends to an automorphism of $G$, thus the automorphism group of $G$ acts on the cohomology group as $\ZZ_q$ acts on its cohomology group, giving three orbits on $\ZZ_q$, and thus six orbits on $\ZZ_2\times\ZZ_q$.
\end{proof}
\section{Equivalence of modular data}
\label{sec:sameST}

We will now determine to what extent the modular data distinguish the centers of the pointed categories of dimension $pq$ listed in the previous section. We will first concentrate on showing that the $p$ pointed fusion categories with noncommutative Grothendieck ring from the start of the previous section have too few distinct modular data to be distinguished by them. This gives the desired counterexample in the paper's title. The main trick is an application of Galois group actions, on the one hand on the cocycles' values, on the other hand in the form of the Galois action defined on modular fusion categories in general.

We will then proceed to investigate to what extent exactly the modular data distinguish centers of pointed fusion categories of dimension $pq$.

Consider the situation of \cref{Lem:pStueck}.
The $p$ different cocycles $\omega^u$ are not so fundamentally different: With the exception of $\omega^0$, any two of them are mapped to each other by an element of the absolute Galois group of abelian extensions $\Gamma:=\Gal(\Q^\ab/\Q)$. In fact there is $\sigma\in\Gamma$ such that for each $i$ we get $\omega^u=\sigma^j\omega$ for some $j\in\{0,...,p-2\}$: it suffices to arrange $\sigma(\zeta_p)=\zeta_p^m$ for a primitive root $m$ modulo $p$. 

It is quite tempting to simply declare as obvious:
\begin{Prop}
  Denote $S^{(u)},T^{(u)}$ the modular data of $\CTR(\Vect_G^{\omega^u})$. Let $\sigma\in\Gamma$ satisfy $\sigma(\zeta_p)=\zeta_p^m$ for $m$ a primitive root modulo $p$. Then there are bijections $p_r$ between the simples of $\CTR(\Vect_G^{\omega})$ and $\CTR(\Vect_G^{\sigma^r\omega})=\CTR(\Vect_G^{\omega^{m^r}})$ such that $S_{p_r(i),p_r(j)}^{(v)}=\sigma^r(S_{ij}^{(1)})$ and $T_{p_r(i)}^{(v)}=\sigma^r(T_{i}^{(1)})$ when $v\equiv m^r(p)$.
\end{Prop}
A rather informal argument would be the following: The definition of $\Vect_G^\omega$ depends on the choice of a $p$-th root of unity involved in the definition of our standard generator of the third cohomology group. Of course, neither the category thus defined, nor anything constructed from it (like the Drinfeld center with its braiding, or the modular data) can crucially depend on that choice. Any other choice will only lead to analogous constructions, with the root of unity replaced in the appropriate way by another root.

We will give a somewhat more precise argument in the following
\begin{proof}
  The simple objects of $\CTR(\Vect_G^\omega)$, which is the module category over the twisted Drinfeld double $D_\omega(G)$, were given in \cite{MR1128130}; the simples are parametrized by pairs $(g,\chi)$, where $g\in G$ represents a conjugacy class, and $\chi$ is an irreducible $\alpha_g$-projective character of the centralizer $C_G(g)$; here $\alpha_g$ is a certain two-cocycle on $C_G(g)$ computed from $\omega$. We can assume that $\chi$ is the character of a projective matrix representation $\rho$ whose matrix coefficients are in a cyclotomic field. When we replace $\omega$ by $\sigma\omega$, then $\alpha_g$ changes to $\sigma\alpha_g$, applying $\sigma$ to the matrix coefficients of $\rho$ yields a $\sigma\alpha_g$-projective representation, whose character is $\sigma\chi$. In \cite{MR1770077}, formulas are given that compute the modular data of $\CTR(\Vect_G^\omega)$ in terms of the group structure, the cocycle $\omega$, and the projective characters representing the simple objects. Now applying $\sigma$ to a matrix coefficient of the $S$- or $T$-matrix amounts to applying $\sigma$ to the values of all the characters and cocycles involved in the formula in \cite{MR1770077}; the bijection $p_r$ therefore just maps the simple corresponding to the pair  $(g,\chi)$ to that corresponding to $(g,\sigma^r\chi)$.
\end{proof}

Next, we consider the Galois action defined for any modular fusion category, originating in \cite{MR1120140,MR1266785}; see \cite[Appendix]{MR2183279}. Since our modular category $\C=\CTR(\Vect^\omega_G)$ is integral, there is, for each $\sigma\in\Gamma$, a unique permutation $\hat\sigma$ of the simple objects of $\C$ such that $\sigma(S_{ij})=S_{i,\hat\sigma(j)}$ (in the general case this is true up to a sign). Thus $\sigma^2(S_{ij})=S_{\hat\sigma(i),\hat\sigma(j)}$. In addition, it is shown in  \cite{MR3435813} that $\sigma^2(T_{ii})=T_{\hat\sigma(i),\hat\sigma(i)}$. 

Combined with the preceding remark, we obtain
\begin{Cor}
  There are at most three different sets among the modular data of the $p$ modular fusion categories $\CTR(\Vect_G^{\omega^u})$. More precisely, the modular data of $\CTR(\Vect_G^{\sigma^r(\omega)})$ only depends on the parity of $r$, so that at least $(p-1)/2$ of these categories share the same modular data.

In particular, for $p>3$ the modular data of the categories $\CTR(\Vect_G^{\omega^u})$ does not distinguish them up to braided equivalence, and for a prime $p>2k$ there are $k$ pairwise braided inequivalent categories among the $\CTR(\Vect_G^{\omega^u})$ that share the same modular data.
\end{Cor}

Note that the smallest example thus obtained is for the nonabelian group of order $55$; the rank of the twisted double is $49$. 

This completes the task of finding the counterexample indicated in the present paper's title. We will give a ``better'' counterexample in the next section, but first we continue to analyze the modular data for completeness:

The group $G=\mathbb{Z}_q \rtimes_n \mathbb{Z}_p$ 
has the presentation
$$G= \langle a,b | a^q = b^p =1 , bab^{-1} = a^n \rangle.$$
There are $\frac{q-1}{p}+p$ conjugacy classes of $G$. A set of representatives of those classes is given by:
$$\lbrace b^k , k \in \{0, \dots , p-1\} \rbrace \cup \lbrace a^l , n^l \equiv 1 \pmod p \rbrace$$
For our calculations we choose a representative for a generator of the third cohomology group
$H^3(\mathbb{Z}_p,\mathbb{C^\times}) \cong \mathbb{Z}_p$, namely \cite[(E.14)]{MooSei:CQCFT}
\begin{equation}\label{eq}
\kappa(\bar{j},\bar{k},\bar{l}):=\exp \left( \frac{2 \pi i}{p^2}[l]([j]+[k]-[j+k])\right)
\end{equation}
where, for a integer $m$, $\bar{m}$ denotes its class in $\mathbb{Z}_p$ and $[m] \in \lbrace 0 , \dots , p-1 \rbrace$ with $m \equiv [m] \pmod p$.

The simple
$D_ \omega(G)$-modules are parametrized by couples $(g, \chi)$ where $g$ is a representative of a conjugacy class in $G$, and $\chi$ is the character of an irreducible $\alpha_g$-projective representation of the centralizer $C_G(g)$, where the $2$-cocycle $\alpha_g$ on $C_G(g)$ is given by: 
$$\alpha_g(x,y):=\omega(g,x,y) \omega^{-1}(x,g,y) \omega(x,y,g)$$
Finally, the element of the $T$-matrix of $D_ \omega(G)-\Mod$ corresponding to a simple $(g,\chi)$ is the scalar $\theta(g,\chi):=\chi(g) / \chi(1_G)$.

\begin{Prop}
For odd prime numbers $p$ and $q$ such that $p | q-1$, there are exactly $3$ different $T$-matrices for the $p$ non-equivalent modular tensor categories $\CTR(\Vect^{\omega^u}_G)$, where $G:=\mathbb{Z}_q \rtimes_n \mathbb{Z}_p$.
\end{Prop}
\begin{proof}
Take first a representative of a conjugacy class of $G$ of the form $a^l$ where $l \in \lbrace 0 , \dots , q-1 \rbrace$ is such that $n^l \equiv 1 \pmod p$. The centralizer in $G$ of this representative is the cyclic group $Z_q$ generated by $a$. So, its third cohomology group is trivial; moreover, as $\omega$ is inflated from the quotient $G/\langle a \rangle$, the cocycle $\alpha_{a^l}$ is trivial on $C_{a^l}(G) \cong Z_q$. Therefore, the $\alpha_{a^l}$-projective representations of $Z_q$, and with them the $\frac{q-1}{p} \times q$ associated elements of the $T$-matrix of $\CTR(\Vect^{\omega^u}_G)$ do not depend on the choice of the $3$-cocycle $\omega^u$ on $G$.

Now, take the trivial element $1_G = b^p \in G$. As all the cocycles we consider are normalized, the $2$-cocycle $\alpha_{1_G}$ on $G$ is trivial and again, the $\frac{q-1}{p}+p$ associated elements of the $T$-matrix do not depend on the choice of the $3$-cocycle on $G$.

So, the only conjugacy classes that may lead to elements of $T$ that depend on the cocycle are the ones represented by $b^k$ with $k \in \lbrace 1 ,\dots , p-1 \rbrace$; there are $p^2$ associated entries of the $T$-matrix. Take such a $k$, and also an integer $u \in \lbrace 1 , \dots , p-1 \rbrace$. We compute explicitly the elements of the $T$-matrix of $\CTR(\Vect^{\omega^u}_G)$ corresponding to the representative $b^k$. The centralizer of $b^k$ in $G$ is the cyclic group $Z_p$ generated by $b$. So, the cocycle $\alpha_{b^k}$ is a coboundary. Take a $1$-cochain $\mu_{b^k}$ such that $\partial \mu_{b^k} = \alpha_{b^k}$. Then the $\alpha_{b^k}$-projective representations of $\ZZ_p$ are usual representations twisted by $\mu_{b^k}$. A generator $\chi_0$ of $\widehat{\langle b \rangle}$ is given by: 
$$ \chi_0(b):=\exp\left(\frac{2\pi i}{p}\right)$$
Take $m \in \lbrace 0 , \dots , p-1 \rbrace$. The $T$-matrix element corresponding to the couple $(b^k, \chi_0^m)$ is given by 
$$ \theta(b^k,\chi_0^m)=\chi_0^m(b^k) \mu_{b^k}(b^k)$$
We need now to compute the scalar $\mu_{b^k}(b^k)$. From \ref{eq}, it is easy to see that $\alpha_{b^k}(x,y)=\omega^u(x,y,{b^k})=\exp \left( \frac{2 \pi i}{p^2}ku([x]+[y]-[x+y])\right)$, and so $\mu_ {b^k}(x)=\exp \left( \frac{2 \pi i}{p^2}ku[x]\right)$. Finally we renumber the simples corresponding to $b^k$, letting $(k,m)\in\{0,\dots,p-1\}^2$ correspond to $(b^k,\chi_0^{m'})$ where $m'k\equiv m(p)$ if $k\not\equiv 0(p)$ and $m'=m$ otherwise. Then
\begin{eqnarray*}
\theta(b^k,\chi_0^{m'}) & = & \chi_0^{m'}(b^k) \mu_{b^k}(b^k)\\
 & = & \exp\left(\frac{2\pi i m'k}{p}\right) \exp\left(\frac{2\pi i k^2 u}{p^2}\right)\\
 & = & \exp\left(\frac{2 \pi i}{p^2}(mp+k^2 u) \right).
\end{eqnarray*}
Thus 
\begin{equation*}
  \{\theta(b^k,\chi_0^{m'})|m=0,\dots,p-1\}=\{\zeta_{p^2}^x|x\equiv k^2u(p)\}=\{z\in\CC|z^p=\zeta_p^{k^2u}\}
\end{equation*}
for each $k\in\{0,\dots,p-1\}$. For $k=0$ this gives the set of $p$-th roots of unity, and the union over the remaining values of $k$ only depends on whether $u$ is a quadratic residue modulo $p$ or not: The values of $\theta(b^k,\chi_0^{m'})$ for $k\in\{1,\dots,p-1\}$ and $m\in\{0,\dots,p-1\}$ are the $p$-th roots of the powers $\zeta_p^x$ where $xu$ is a quadratic residue modulo $p$; each value occurs exactly twice.
To conclude, we just need to observe that for the trivial cocycle, all the $\mu_g$ are trivial for the representative $g$ of any conjugacy class of $G$.
\end{proof}

\begin{Cor}
The twisted doubles of groups of order $3q$ with a prime $q>3$ are classified by their modular data.
\end{Cor}

\section{More counterexamples}
\label{sec:worse}

The module categories of twisted Drinfeld doubles are fusion categories defined over a cyclotomic field, and for such categories a Galois twist can be defined \cite{2013arXiv1305.2229D}; clearly the categories found above which share the same modular data are Galois twists of each other. One might be tempted to immediately refine the refuted conjecture: Perhaps modular fusion categories defined over a cyclotomic field can only share the same modular data if they are, at worst, Galois twists of each other?

Consider three primes $p<q<r$ such that $p$ divides both $q-1$ and $r-1$. Let $G_q$ be the semidirect product group considered above, define $G_r$ in the same way, and consider $G=G_q\times G_r$. Further define $\omega_{uv}:=\Inf_{G_q}^G\Inf_{\ZZ_p}^{G_q}(\kappa^u)\cdot\Inf_{G_r}^G\Inf_{\ZZ_p}^{G_r}\kappa^v$.

The same arguments as above show that the categories $\C_{uv}:=\Vect_G^{\omega_{uv}}$ are pairwise not categorically Morita equivalent (Again, any automorphism leaves $\omega_{uv}$ invariant, and the only abelian normal subgroups are $\ZZ_q$, $\ZZ_r$, and $\ZZ_q\times\ZZ_r$).

The modular data for $\CTR(\Vect_G^{\omega_{uv}})$ is the Kronecker product of the modular data for $\CTR(\Vect_{G_q}^{\omega^u})$ and $\CTR(\Vect_{G_r}^{\omega^v})$. Thus, there are at most $3\cdot 3=9$ different modular data. But the effect of a Galois automorphism on $\omega_{uv}$ is a cyclic permutation of $\{1,\dots,p-1\}$ which is to be applied to $u$ and $v$ simultaneously (unless $u$ or $v$ is zero). Thus, there are $p+1$ orbits of cocycles, and thus $p+1$ categories among the $\CTR(\Vect_G^{\omega_{uv}})$ that are pairwise inequivalent, even up to a Galois twist.

Of course one may argue that the inequivalent up to Galois twist examples are thus obtained due to a cheap trick: They are Deligne products of two modular categories, and we simply have a separate Galois twist on each factor.
\begin{Qu}
  If two modular categories defined over the cyclotomic numbers are prime in the sense of \cite{MR1990929}, and if they share the same modular data, are they then Galois twists of each other? Is the module category of a twisted Drinfeld double $D_\omega(G)$ of an indecomposable group determined up to Galois twist by its modular data?
\end{Qu}

\bibliographystyle{alpha}

\bibliography{eigene,andere,arxiv,mathscinet}

\def\cprime{$'$} \def\cprime{$'$} \def\germ{\mathfrak}\def\cprime{$'$}
  \def\cfgrv#1{\ifmmode\setbox7\hbox{$\accent"5E#1$}\else
  \setbox7\hbox{\accent"5E#1}\penalty 10000\relax\fi\raise 1\ht7
  \hbox{\lower1.05ex\hbox to 1\wd7{\hss\accent"12\hss}}\penalty 10000
  \hskip-1\wd7\penalty 10000\box7} \def\cprime{$'$}
  \def\cfgrv#1{\ifmmode\setbox7\hbox{$\accent"5E#1$}\else
  \setbox7\hbox{\accent"5E#1}\penalty 10000\relax\fi\raise 1\ht7
  \hbox{\lower1.05ex\hbox to 1\wd7{\hss\accent"12\hss}}\penalty 10000
  \hskip-1\wd7\penalty 10000\box7} \def\cprime{$'$}
  \def\cfgrv#1{\ifmmode\setbox7\hbox{$\accent"5E#1$}\else
  \setbox7\hbox{\accent"5E#1}\penalty 10000\relax\fi\raise 1\ht7
  \hbox{\lower1.05ex\hbox to 1\wd7{\hss\accent"12\hss}}\penalty 10000
  \hskip-1\wd7\penalty 10000\box7}
\begin{thebibliography}{BNRW16}

\bibitem[BK01]{BakKir:LTCMF}
Bojko Bakalov and Jr.~Alexander Kirillov.
\newblock {\em Lectures on tensor categories and modular functors}, volume~21
  of {\em University Lecture Series}.
\newblock American Mathematical Society, Providence, RI, 2001.

\bibitem[BNRW16]{MR3486174}
Paul Bruillard, Siu-Hung Ng, Eric~C. Rowell, and Zhenghan Wang.
\newblock Rank-finiteness for modular categories.
\newblock {\em J. Amer. Math. Soc.}, 29(3):857--881, 2016.

\bibitem[Bro82]{MR672956}
Kenneth~S. Brown.
\newblock {\em Cohomology of groups}, volume~87 of {\em Graduate Texts in
  Mathematics}.
\newblock Springer-Verlag, New York-Berlin, 1982.

\bibitem[CE56]{MR0077480}
Henri Cartan and Samuel Eilenberg.
\newblock {\em Homological algebra}.
\newblock Princeton University Press, Princeton, N. J., 1956.

\bibitem[CG94]{MR1266785}
A.~Coste and T.~Gannon.
\newblock Remarks on {G}alois symmetry in rational conformal field theories.
\newblock {\em Phys. Lett. B}, 323(3-4):316--321, 1994.

\bibitem[CGR00]{MR1770077}
Antoine Coste, Terry Gannon, and Philippe Ruelle.
\newblock Finite group modular data.
\newblock {\em Nuclear Phys. B}, 581(3):679--717, 2000.

\bibitem[dBG91]{MR1120140}
Jan de~Boer and Jacob Goeree.
\newblock Markov traces and {${\rm II}_1$} factors in conformal field theory.
\newblock {\em Comm. Math. Phys.}, 139(2):267--304, 1991.

\bibitem[DHW13]{2013arXiv1305.2229D}
O.~{Davidovich}, T.~{Hagge}, and Z.~{Wang}.
\newblock {On Arithmetic Modular Categories}.
\newblock {\em ArXiv e-prints}, May 2013.

\bibitem[DLN15]{MR3435813}
Chongying Dong, Xingjun Lin, and Siu-Hung Ng.
\newblock Congruence property in conformal field theory.
\newblock {\em Algebra Number Theory}, 9(9):2121--2166, 2015.

\bibitem[DPR90]{MR1128130}
R.~Dijkgraaf, V.~Pasquier, and P.~Roche.
\newblock Quasi {H}opf algebras, group cohomology and orbifold models.
\newblock {\em Nuclear Phys. B Proc. Suppl.}, 18B:60--72 (1991), 1990.
\newblock Recent advances in field theory (Annecy-le-Vieux, 1990).

\bibitem[EGNO15]{MR3242743}
Pavel Etingof, Shlomo Gelaki, Dmitri Nikshych, and Victor Ostrik.
\newblock {\em Tensor categories}, volume 205 of {\em Mathematical Surveys and
  Monographs}.
\newblock American Mathematical Society, Providence, RI, 2015.

\bibitem[ENO05]{MR2183279}
Pavel Etingof, Dmitri Nikshych, and Viktor Ostrik.
\newblock On fusion categories.
\newblock {\em Ann. of Math. (2)}, 162(2):581--642, 2005.

\bibitem[GMN07]{MR2333187}
Christopher Goff, Geoffrey Mason, and Siu-Hung Ng.
\newblock On the gauge equivalence of twisted quantum doubles of elementary
  abelian and extra-special 2-groups.
\newblock {\em J. Algebra}, 312(2):849--875, 2007.

\bibitem[ML95]{MR1344215}
Saunders Mac~Lane.
\newblock {\em Homology}.
\newblock Classics in Mathematics. Springer-Verlag, Berlin, 1995.
\newblock Reprint of the 1975 edition.

\bibitem[MS89]{MooSei:CQCFT}
Gregory Moore and Nathan Seiberg.
\newblock Classical and quantum conformal field theory.
\newblock {\em Comm. Math. Phys.}, 123(2):177--254, 1989.

\bibitem[M{\"u}g03]{MR1990929}
Michael M{\"u}ger.
\newblock On the structure of modular categories.
\newblock {\em Proc. London Math. Soc. (3)}, 87(2):291--308, 2003.

\bibitem[Nai07]{MR2362670}
Deepak Naidu.
\newblock Categorical {M}orita equivalence for group-theoretical categories.
\newblock {\em Comm. Algebra}, 35(11):3544--3565, 2007.

\bibitem[Nik13]{MR3077244}
Dmitri Nikshych.
\newblock Morita equivalence methods in classification of fusion categories.
\newblock In {\em Hopf algebras and tensor categories}, volume 585 of {\em
  Contemp. Math.}, pages 289--325. Amer. Math. Soc., Providence, RI, 2013.

\bibitem[Ost03]{MR1976459}
Victor Ostrik.
\newblock Module categories, weak {H}opf algebras and modular invariants.
\newblock {\em Transform. Groups}, 8(2):177--206, 2003.

\end{thebibliography}

\end{document}